\documentclass [12pt]{amsart}

\usepackage{amscd,amsmath,amssymb,euscript}
\usepackage[frame,cmtip,curve,arrow,matrix,line,graph]{xy}

\usepackage{etex}

\usepackage{tikz}
\usetikzlibrary{decorations.pathreplacing}
\usetikzlibrary{patterns}

\date{}
\newtheorem{thm}{Theorem}[section]
\newtheorem{theorem}{Theorem}
\newtheorem{lemma}[thm]{Lemma}
\theoremstyle{definition}
\newtheorem{definition}[thm]{Definition}
\newtheorem{defn}[thm]{Definition}

\newtheorem{remarks}[thm]{Remarks}
\newtheorem{remark}[thm]{Remark}

\newtheorem{examples}[thm]{Examples}

\renewcommand{\leq}{\leqslant}
\renewcommand{\geq}{\geqslant}
\newcommand{\onto}{\twoheadrightarrow}

\newcommand{\E}{\mathbb{E}}
\renewcommand{\mod}{\operatorname{mod}}

\renewcommand{\Im}{\operatorname{Im}}

\newcommand{\D}{D}

\newcommand{\Coh}{\operatorname{Coh}}
\newcommand{\Aut}{\operatorname{Aut}}
\newcommand{\isom}{\cong}

\newcommand{\tensor}{\otimes}
\newcommand{\im}{\operatorname{im}}

\newcommand{\C}{\mathbb C}
\newcommand{\TT}{\mathcal T}
\newcommand{\RR}{\mathbb{R}}

\renewcommand{\im}{\operatorname{im}}
\newcommand{\coker}{\operatorname{coker}}

\newcommand{\F}{\mathbb F}
\newcommand{\Z}{\mathbb Z}
\newcommand{\A}{\mathcal A}

\newcommand{\CC}{\mathcal C}
\newcommand{\OO}{\mathcal O}

\newcommand{\id}{\operatorname{id}}
\newcommand{\Ext}{\operatorname{Ext}}
\newcommand{\Hom}{\operatorname{Hom}}

\newcommand{\lRa}[1]{\xrightarrow{\ #1\ }}
\newcommand{\lra}{\longrightarrow}
\newcommand{\R}{{\RR}}

\newcommand{\FF}{\mathcal{F}}
\newcommand{\ch}{{\rm ch}}

\newcommand{\M}{{\mathcal{M}}}
\renewcommand{\O}{\OO}
\newcommand{\Hilb}{\operatorname{Hilb}}
\newcommand{\PHilb}{\operatorname{Pairs}}
\newcommand{\Quot}{\operatorname{Quot}}
\newcommand{\an}{\operatorname{an}}
\newcommand{\naive}{naive}
\newcommand{\DT}{\operatorname{DT}}
\newcommand{\ET}{\operatorname{{DT}^{\naive}}}
\newcommand{\qET}{\operatorname{{q-DT}^{\naive}}}
\newcommand{\PT}{\operatorname{PT}}
\newcommand{\EP}{\operatorname{{PT}^{\naive}}}

\newcommand{\Obj}{\operatorname{Obj}}
\newcommand{\Q}{\mathbb{Q}}
\newcommand{\GL}{\operatorname{GL}}

\newcommand{\Gr}{\operatorname{Gr}}

\newcommand{\<}{\langle}
\renewcommand{\>}{\rangle}

\newcommand{\ahom}{\dim_\C \Hom_\A}
\newcommand{\aext}{\dim_\C \Ext_\A}

\newcommand{\X}{\mathcal{X}}

\newcommand{\St}{\operatorname{St}}
\renewcommand{\S}{\mathbb{S}}

\newcommand{\reg}{\operatorname{reg}}
\newcommand{\Hall}{\operatorname{Hall}}
\newcommand{\mot}{{\rm mot}}
\newcommand{\fin}{{\rm fty}}
\newcommand{\hHall}{{\Hall}^{\,\wedge}}
\newcommand{\Spec}{\operatorname{Spec}}
\newcommand{\Stab}{\operatorname{Stab}}

\renewcommand{\P}{\mathcal{P}}

\newcommand{\IH}{\mathbb{H}}
\newcommand{\T}{\mathbb{T}}

\renewcommand{\ss}{{\rm ss}}

\newcommand{\Per}{\operatorname{Per}}
\newcommand{\Var}{\operatorname{Var}}
\newcommand{\I}{{\mathcal I}}

\newcommand{\cT}{\mathcal{T}}

\newcommand{\qDT}{\operatorname{q-DT}}

\begin{document}

 %\nocite{*}

\title{ Hall algebras  and Donaldson-Thomas invariants}
\author{Tom Bridgeland}
\maketitle

\begin{abstract}This is a survey article about Hall algebras and their applications to the 
study of motivic  invariants of moduli spaces of coherent sheaves on Calabi-Yau threefolds. The ideas presented here are mostly due to  Joyce, Kontsevich, Reineke, Soibelman and Toda.
\end{abstract}

\section{Introduction}

%\subsection{Summary}
Our aim in this article is to give a brief introduction to Hall algebras, and explain how they can be used to study motivic invariants of moduli spaces of coherent sheaves on Calabi-Yau threefolds. In particular, we discuss  generalized Donaldson-Thomas (DT) invariants, and  the Kontsevich-Soibelman wall-crossing formula, which describes their behaviour   under variations of stability parameters. Many long and  difficult papers have been written on these topics:  here we focus on the most basic aspects of the story, and give pointers to the literature.

We begin our introduction to Hall algebras in Section 2. 
In this introductory section we will try to motivate the reader by discussing some of the more concrete applications. The theory we shall describe applies quite generally to motivic invariants of moduli spaces of sheaves on Calabi-Yau threefolds, but some of the most striking results relate to curve-counting invariants, and for the sake of definiteness we will  focus on these.%For the most part we will be following the work of  M. Reineke, D. Joyce, M. Kontsevich and Y. Soibelman, and Y. Toda.
%In a nutshell we are interested in
%\begin{itemize}
%\item[(a)] Calculating motivic invariants of moduli spaces of coherent sheaves on Calabi-Yau threefolds, e.g. DT invariants.%, e.g. in the Calabi-Yau threefold case,  Donaldson-Thomas invariants.
%\smallskip
%
%\item[(b)] Understanding the dependence of these invariants on the stability parameters.
%\end{itemize}
\subsection{Motivic invariants}
Since the word motivic has rather intimidating connotations in general,  let us make clear from the start that in this context it simply refers to invariants of varieties which have the property that
\[\chi(X)=\chi(Y)+\chi(U),\]
whenever $Y\subset X$ is a closed subvariety and $U=X\setminus Y$.
A good example is the Euler characteristic: if $X$ is a variety over $\C$ we can define
\[e(X)=\sum_{i\in \Z} (-1)^i \dim_\C H^i(X^{\rm an},\C)\in \Z,\]
where the cohomology groups are the usual singular cohomology groups of  $X$  equipped with the analytic  topology.

Of crucial importance for the theory we shall describe is Behrend's  discovery \cite{behrend} of the motivic nature of DT invariants. 
If $M$ is a fine projective moduli scheme parameterizing stable coherent sheaves on a Calabi-Yau threefold $X$, there is a corresponding DT invariant \cite{thomas}
\[\DT(\M)=\int_{M^{\rm vir}} 1\in \Z,\] defined to be the degree of the virtual fundamental class of $M$. Behrend proved  that this invariant can also be computed as a weighted Euler characteristic
\[\DT(\M)=e(M;\nu):=\sum_{n\in \Z} n\cdot e(\nu^{-1}(n))\in \Z,\]
where $\nu\colon M\to \Z$  is a certain constructible function,
 depending only on the singularities of the scheme $M$.

Surprisingly, it turns out that for most of the applications described below one can equally well consider naive DT (or `Euler-Thomas') invariants
\[\ET(\M)=e(M)\in \Z,\]
 and the reader unfamiliar with virtual fundamental classes and the Behrend function will not miss anything by restricting to this case. Nonetheless, the genuine invariants are more important for several reasons:  they are unchanged by deformations of $X$, they have subtle integrality properties, and they are directly relevant to physics.

\subsection{Example: Toda's flop formula}

%{Curve-counting invariants}
%Suppose $X$ is a Calabi-Yau threefold.

Let $X$ be a  smooth projective Calabi-Yau threefold over $\C$. We always take this to include the condition that \[H^1(X,\O_X)=0.\]
Fix  $\beta\in H_2(X,\Z)$ and $n\in\Z$ and consider the Hilbert scheme
%\[\Hilb(\beta,n)=\big\{\stackrel{\text{closed subschemes }C\subset X\text{ of dimension $\leq 1$}}{\text{ with } [C]=\beta, \chi(\OO_C)=n}\big\}\]
\[\Hilb(\beta,n)=\bigg\{\parbox{17em}{\centering closed subschemes $C\subset X$ of dim $\leq 1$ \\ \smallskip satisfying $[C]=\beta$ and $\chi(\OO_C)=n$}\bigg\}.\]
This can be viewed as a fine moduli space for rank one torsion-free sheaves on $X$ by mapping $C\subset X$ to its ideal sheaf $\I_C$ (at the level of $\C$-valued points this identification is easy, see e.g. \cite[Lemma 2.2]{bridgeland_maciocia}, and for the motivic  statements here this suffices; the full scheme-theoretic isomorphism is covered in \cite[Section 2]{pandharipande_thomas1}).  We can  consider the corresponding naive DT  invariants
\[\ET(\beta,n)=e(\Hilb(\beta,n))\in \Z,\]
or by introducing the Behrend function, their more genuine cousins.

%The results we shall describe below apply equally well to these naive invariants as to the genuine DT invariants, so the reader unfamiliar with virtual fundamental classes can happily stick with this case. Nonetheless, the genuine invariants are more important because they are unchanged by deformations of $X$, and because they have subtle integrality properties.
%
% The genuine DT invariants $\DT(\beta,n)$ are defined as the degree of a virtual fundamental class which is a reflection of the natural derived structure on $\Hilb(\beta,n)$ viewed as a moduli space of ideal sheaves. Behrend proved that these invariants are also motivic: there is a constructible function
% \[\nu\colon \Hilb(\beta,n)\to \Z\] depending only on the singularities of the scheme $\Hilb(\beta,n)$ such that
% \[\DT(\beta,n)=e(\Hilb(\beta,n);\nu):=\sum_{n\in \Z} n\cdot e(\nu^{-1}(n))\in \Z,\]
% is  the corresponding weighted Euler characteristic.
% It is this result that allows the application of motivic techniques to the study of DT invariants.

 Let us now consider  two smooth projective Calabi-Yau threefolds $X_\pm$ related by a flop:
\[\xymatrix@C=1.4em{ X_+\ar[dr]_{f_+}&&
X_-\ar[dl]^{f_-} \\
&Y }\]
It seems very natural to ask how the DT invariants are affected by this birational transformation. 

\begin{theorem}[Toda, \cite{toda6}]
\label{toda1}
The expression
\[{\frac{\sum_{(\beta,n)} \ET(\beta,n)\,  x^\beta y^n}{\sum_{(\beta,n):
f_*(\beta)=0} \ET(\beta,n) \, x^\beta y^n}}\] is the same on both
sides of the flop, after making the natural identification
\[H_2(X_+,\Z) \isom H_2(X_-,\Z)\]
induced by strict transform of divisors.
 \end{theorem}
%\smallskip

%This result was extended to genuine DT invariants by Calabrese.
%\smallskip 

The result was extended to genuine DT invariants using a different argument by Calabrese \cite{calabrese1}, and Toda's argument now also applies to this case \cite{toda7}. In the case when the flopped curves have normal bundle $\O(-1)^{\oplus 2}$ the result was  proved earlier by Hu and Li \cite{huli} using different techniques.% Note that although we need to assume that $X_\pm$ are smooth, we do not need to assume anything about the normal bundles of the flopped curves.

\subsection{Example: the DT/PT correspondence}
\label{poop}
Pandharipande and Thomas \cite{pandharipande_thomas1} introduced an `improved' version of the moduli space $\Hilb(\beta,n)$ which eliminates  the problem of free-roaming points.
 A stable pair on $X$ is a map
\[f\colon \OO_X \to
E\] of coherent sheaves
 such that
 \[\text{(a) $E$ is pure of dimension 1,}\qquad \text{(b)} \dim \operatorname{supp} \coker(f)=0.\]
%\begin{itemize}
%\item[(a)] $E$ is pure of dimension 1,\smallskip
%
%  \item[(b)]$\dim_\C \operatorname{supp} \coker(f)=0$.\smallskip
%\end{itemize}

 Fixing 
a class $\beta\in H_2(X,\Z)$ and $n\in \Z$ as before, there is a fine moduli scheme $\PHilb(\beta,n)$ parameterizing stable pairs with $\ch(E)=(0,0,\beta,n)$.
We  can then consider naive stable pair invariants \[\EP(\beta,n)=e(\PHilb(\beta,n))\in \Z.\]
Genuine stable pair invariants  are obtained by weighting with the Behrend function as before.

%Stable pair invariants were introduced with the aim of reproducing the reduced DT invariants appearing in the DT/GW correspondence.

\begin{theorem}[Toda, \cite{toda2}]
\label{toda2}
\begin{itemize}
\smallskip

\item[(i)]For each $\beta\in H_2(X,\Z)$ there is an identity
%\vspace{-0.3cm}
\[{\sum_{n\in \Z} \EP(\beta,n) y^n = \frac{\sum_{n\in \Z} 
\ET(\beta,n)y^n}{\sum_{n\geq 0}\ET(0,n) y^n}}.\]
\item[(ii)]
This formal power series is the  Laurent expansion of a rational function of $y$, invariant under $y\leftrightarrow y^{-1}$.
\end{itemize}
\end{theorem}
\smallskip

These results have since been shown to hold for genuine invariants \cite{bridgeland_curve-counting, toda7}. Part (i) had previously been conjectured by Pandharipande and Thomas \cite[Sect. 3]{pandharipande_thomas1}; part (ii) then becomes 
part of the famous MNOP conjectures \cite[Conj. 2]{mnop}. See also \cite{stoppa_thomas} for a generalization of Theorem \ref{toda2} to arbitrary threefolds. %The denominator  on the RHS of (i) is known explicitly: it is the Taylor expansion of a power of the MacMahon function:
%\[\sum_{n\geq 0}\ET(0,n) y^n=\bigg(\prod_{n\geq 0} (1-y^n)^{n}\bigg)^{\chi(M)}.\]
%Part (ii) is important for the  DT/GW correspondence \cite{mnop, pandharipande_pixton}, which involves setting $y=-e^{iu}$ and re-expanding about $u=0$.

\subsection{General strategy}
The basic method for proving the above results (and many more like them) is taken from Reineke's work on 
the cohomology groups of moduli spaces of quiver representations \cite{reineke2}. One can thus view the whole subject as a showcase for the way in which techniques pioneered in the world of representations of quivers can solve important problems in algebraic geometry.  
The strategy consists of three steps:
%The basic strategy for proving these results is as follows:%\footnote{this goes back to 
%Reineke's work on }
%\smallskip

\begin{itemize}

\item[(a)] Describe the relevant phenomenon in terms of wall-crossing: a change of stability condition  in an abelian or triangulated category $\CC$. \smallskip

\item[(b)] Write down an appropriate identity in the Hall algebra of  $\CC$.\smallskip

\item[(c)] Apply a ring homomorphism $\I\colon \Hall(\CC) \to \C_q[K_0(\CC)]$
to obtain the required identity of generating functions.

\end{itemize}

 The first two steps are completely general, but the existence of the map $\I$ (known as the integration map) requires either
 \smallskip
 
\begin{itemize}
\item[(i)] $\CC$ is hereditary: 
$\Ext_{\CC}^{i}(M,N)=0$ for $i>1$,
%\[\aext^1(M,N) -  \ahom(M,N) =\chi(M,N),\]
\smallskip

\item[(ii)] $\CC$ satisfies the CY$_3$ condition:
$\Ext_{\CC}^i(M,N)\isom \Ext_{\CC}^{3-i}(N,M)^*$.
%\[( \aext^1(M,N) -\ahom(M,N))-( \aext^1(N,M) -  \ahom(N,M))=\chi(M,N).\]
\end{itemize}

Hall algebras and an example of a Hall algebra identity will be introduced in Section 2. Integration maps are discussed in Section 3. The most basic wall-crossing identity, resulting from the existence and uniqueness of Harder-Narasimhan filtrations, will be discussed in Section 4. The application of the above general strategy to Theorems \ref{toda1} and \ref{toda2}  will be explained in Section 5.

\subsection
{Some history}

It is worth noting the following  pieces of pre-history which provided essential ideas for the results described here.

\begin{itemize}

\item[(1)]Computation of the Betti numbers of moduli spaces of semistable bundles on curves  using the Harder-Narasimhan stratification   (Harder-Narasimhan \cite {harder_narasimhan}, Atiyah-Bott \cite{atiyah_bott}).\smallskip

\item[(2)]Wall-crossing behavior of  moduli spaces with parameters, e.g. work of Thaddeus \cite{thaddeus}  on moduli of stable pairs on curves.\smallskip

\item[(3)]Use of derived categories and changes of t-structure to increase the flexibility of wall-crossing techniques, e.g. threefold flops \cite{bridgeland_flop}.\smallskip

\item[(4)]Systematic use of Hall algebras: Reineke's calculation of Betti numbers of moduli spaces of representations of quivers \cite{reineke2}.\smallskip

\item[(5)] Behrend's interpretation of Donaldson-Thomas invariants as weighted Euler characteristics \cite{behrend}.

\end{itemize}

The credit for the development of motivic Hall algebras as a tool for studying moduli spaces of sheaves on Calabi-Yau threefolds is due jointly to Joyce and to Kontsevich and Soibelman. 
Joyce introduced motivic Hall algebras in a long series of papers \cite{joyce1,joyce2,joyce4,joyce6,joyce5,joyce7}. He used this framework to define generalizations of the naive Donaldson-Thomas invariants considered above, which apply to moduli stacks containing strictly semistable sheaves. He also  worked out  the wall-crossing formula for these invariants and proved a very deep no-poles theorem. Kontsevich and Soibelman \cite{kontsevich_soibelman1} constructed an alternative theory which incorporates  motivic vanishing cycles, and therefore applies to genuine DT invaraints and motivic versions thereof. They also produced a more  conceptual statement of the wall-crossing formula. Some of their work was conjectural and is still being developed today. Joyce and Song \cite{joyce_song} later  showed how to directly incorporate the Behrend function into Joyce's  framework, and so obtain rigorous results on DT invariants.

\smallskip

{\bf Notes.} There are quite a few survey articles on the topics covered here. For a survey of curve-counting invariants we recommend \cite{pandharipande_thomas_survey}. Joyce \cite{joyce8} and Kontsevich-Soibelman \cite{kontsevich_soibelman_survey} produced surveys of their work in this area. Toda \cite{toda_survey} also wrote a survey of wall-crossing techniques in DT theory.

%{Example 2: DT/PT correspondence}
%
%\begin{itemize}
%\item[(a)] Toda proved that for each $\beta$ there is an identity
%\[{\sum_{n\in \Z} \EP(\beta,n) y^n = \frac{\sum_{n\in \Z} 
%\ET(\beta,n)y^n}{\sum_{n\geq 0}\ET(0,n) y^n}}\]
%
%%The same result  holds for genuine DT invariants.
%\item[(b)]Toda also used wall-crossing techniques to prove a conjecture of [MNOP], namely that the above formal power series is the  Laurent expansion of a rational function of $y$, invariant under $y\leftrightarrow y^{-1}$.
%\smallskip
%%
%
%\item[(c)]These results also hold for genuine DT invariants.
%%
%%\item[(c)]
%\end{itemize}
%%\end{example}
%%
%

%%%%%%%%%%%%%%%%%%%%%%%%%%%%%%%%%%%%%%%%%%%%%%%

%%%%%%%%%%%%%%%%%%%%%%%%%%%%%%%%%%%%%%%%%%%%%%%

%%%%%%%%%%%%%%%%%%%%%%%%%%%%%%%%%%%%%%%%%%%%%%%

\section{Hall algebras}

The aim of this section is to introduce the  idea of a  Hall algebra in general, and introduce the particular kind `motivic Hall algebras' which will be important for our applications to moduli spaces. As a warm-up we  begin by discussing finitary Hall algebras. From our point-of-view these are rather simplified models, but one of the important features of this subject is that `back-of-the-envelope' calculations can be easily  made in the finitary case before being generalized to the more realistic motivic setting.

\subsection{Finitary Hall algebras}
\label{start}

Suppose that $\A$ is an essentially small abelian category satisfying the following strong finiteness conditions:

\smallskip

\begin{itemize}
\item[(i)] Every object has only finitely many subobjects.\smallskip

\item[(ii)] All  groups $\Ext^i_\A(E,F)$ are finite.
\end{itemize}

Of course these conditions are never satisfied for categories of coherent sheaves but there are nonetheless plenty of examples: let $A$ be any finite dimensional  algebra  over  a finite field $k=\mathbb{F}_q$, and take  $\A=\mod(A)$ 
to be the  category of finite dimensional left $A$--modules.

%Under these conditions one can make the following definitions.

\begin{defn}
The finitary Hall algebra of $\A$ is defined to be the set of all complex-valued functions on isomorphism classes of $\A$
\[\hHall_{\fin}(\A)=\big\{f\colon (\Obj(\A)/{\scriptstyle \isom})  \lra \C\big\},\]
 equipped with a convolution product coming from short exact sequences:
\[(f_1*f_2)(B)=\sum_{A\subset B} f_1(A) \cdot f_2(B/A).\]
This is an associative, but usually non-commutative, unital algebra. We also define a subalgebra \begin{equation}
\label{multiple}\Hall_\fin(\A)\subset \hHall_\fin(\A),\end{equation} consisting of functions with finite support.
\end{defn}

Before going further the reader should prove that the Hall product indeed gives an associative multiplication,  and that  multiple products are given by the formula\[(f_1*\cdots *f_n)(M)=\sum_{0=M_0\subset M_1\subset \cdots \subset M_n= M} f_1(M_1/M_0) \cdots f_n(M_n/M_{n-1}).\]
Finally one should check that the the characteristic function $\delta_0$ of the zero object is the multiplicative unit. 

For each object $E\in \A$ we consider an element $\delta_E\in \Hall_\fin(\A)$ which is the characteristic function of the isomorphism class of $E$, and the closely related element
\[\kappa_E=|\Aut(E)|\cdot \delta_E\in \Hall_\fin(\A).\]
The following Lemma was first proved by Riedtmann.

\begin{lemma}
\label{ried}For any objects $A,C\in \A$ we have an identity

  \[\kappa_A* \kappa_C= \sum_{B\in \A} \frac{|\Ext^1(C,A)_B|}{|\Hom(C,A)|}\cdot \kappa_B,\]
where 
$\Ext^1(C,A)_B\subset \Ext^1(C,A)$
denotes the subset of extensions whose  middle term is isomorphic to $B$. 
\end{lemma}

\begin{proof}
This is another very good exercise. See   \cite[Lemma 1.2]{schiffmann_notes}.
\end{proof}

One more piece of notation: we define an element $\delta_\A\in \hHall_\fin(\A)$ by setting
\[\delta_\A(E)=1\quad \text{ for all } E\in \A.\]
This should not be confused with the identity element $1=\delta_0\in \hHall_\fin(\A)$.
%We remark that there was nothing special about taking complex-valued functions in our definition of the Hall algebra: we could also take integral or rational functions.

\subsection{Example: category of vector spaces}
\label{vect}

Let $\A=\operatorname{Vect_k}$ be the category of  finite dimensional  vector spaces over $\F_q$.
Let
\[\delta_n\in \Hall_\fin(\A)\]
denote the characteristic function of vector spaces of dimension $n$. The definition immediately gives
\[\delta_n * \delta_m= |\Gr_{n,n+m}(\F_q)| \cdot \delta_{n+m}.\]
The number of $\F_q$-valued points of the Grassmannian appearing here is easily computed: it is the $q$-binomial coefficient
\[|\Gr_{n,n+m}(\F_q)|= \frac{(q^{n+m}-1) \cdots (q^{m+1}-1)}{(q^n -1)\cdots (q-1) } = \binom{n+m}{n}_q. \]
It then follows that there is an isomorphism of algebras
 \[\I\colon \Hall_\fin(\A) \to \C[x], \qquad
\I(\delta_n)= \frac{q^{n/2}\cdot x^n}{(q^n-1)\cdots (q-1)},\]
where the factor $q^{n/2}$ is inserted for later convenience.
This is in fact a first example of an integration map: in this special case it is an isomorphism, because the isomorphism class of an object of $\A$ is completely determined by its numerical invariant $n\in \Z_{\geq 0}$.

The isomorphism $\I$ maps  the element $\delta_\A=\sum_{n\geq 0}\delta_n$  to the series 
\[\E_q(x)=\sum_{n\geq 0} \frac{q^{n/2} \cdot x^n}{(q^n-1) \cdots (q-1)}\in \C[[x]].\]
This series is known as the quantum dilogarithm \cite{fock_goncharov,kashaev,keller}, because if we view $q$ as a variable, then 
\[\log \E_q(x) = \frac{-1}{(q-1)}\cdot \sum_{n\geq 1} \frac{x^n}{n^2} + O(1),\]
as $q^{1/2}\to -1$ in the region  $|q|<1$. 
This identity will be very important later: it gives rise to the multiple cover formula in Donaldson-Thomas theory.
%Kashaev `The q-binomial formula ...'

\subsection{Quotient identity}
\label{id}

The beauty of the Hall algebra construction  is the way that it allows one to turn categorical statements into algebraic identities. As we shall see in Sections \ref{stab} and \ref{wc} (which can also be read now),  this is the basis for the Kontsevich-Soibelman wall-crossing formula. Here we give a different  example, which is the basis of our approach to  Theorems \ref{toda1} and \ref{toda2}.
%\smallskip

Let $\A$ be an abelian category satisfying the finiteness assumptions as above, and let us also fix an object $P\in \A$. Introduce  elements  \[\delta_\A^P\in \hHall_\fin(\A), \quad  \Quot^P_\A\in \hHall_\fin(\A),\]
by defining, for any object $E\in \A$,
\[\delta^P_\A(E)=|\Hom_\A(P,E)|, \quad \Quot^P_\A(E)=|\Hom_\A^{\onto}(P,E)|,\]
where $\Hom^\onto_\A(P,E)\subset \Hom_\A(P,E)$ is the subset of surjective maps.
The following is a variant of \cite[Lemma 5.1]{engel_reineke}.
%\[\Hom^\onto_\A(P,M)=\{f\colon P\to M: f \text{ surjective}\}\subset \Hom_\A(P,M).\]

\begin{lemma}
\label{idid}
There is an identity
\[{\delta_\A^P = \Quot^P_\A * \,\delta_\A}\]
in the Hall algebra $\hHall_\fin(\A)$.
\end{lemma}

\begin{proof}
Evaluating on an object $E\in \A$ gives
\[|\Hom_\A(P,E)|=\sum_{A\subset E} |\Hom_\A^\onto(P,A)|\cdot 1,\]
which holds because every map $f\colon P\to E$ factors uniquely via its image.
\end{proof}

It is a  fun exercise to apply this result in the case when $\A=\operatorname{Vect_k}$ and $P=k^{\oplus d}$, to obtain an identity involving the quantum dilogarithm $\E_q(x)$.
%We will discuss a version of the same identity in the world of motivic Hall algebras in Section \ref{motquot} below.%;  this  is the crucial ingredient in one approach to Theorems \ref{toda1} and \ref{toda2}.%For now, let us note that taking $\A=\operatorname{Vect}$ as in the last section, and $P=k^{\oplus d}$ we the identity
%\[\Phi_q(x q^d)=(1+x)^d \cdot \Phi_q(x).\]

\subsection{Hall algebras in general}
\label{general}

A given abelian  category $\A$ may have many different flavours of Hall algebra
associated to it: finitary Hall algebras, Hall  algebras of constructible functions, motivic Hall algebras, cohomological Hall algebras, etc.
In this  section we shall make some general (and intentionally vague) remarks relevant to any of these: our point-of-view is that the different types of Hall algebra should be thought of as different ways to take the `cohomology' of the moduli stack of objects of $\A$. 

For definiteness  we take $\A$ to be the category of coherent sheaves  on a smooth projective variety $X$. %, although  we could equally well consider the category of finite-dimensional representations of an algebra of finite global dimension. 
Consider the stack $\M$  of objects of $\A$, and the stack $\M^{(2)}$  of short exact sequences in $\A$. There is a diagram of morphisms of stacks
\begin{equation}
\label{diagram}\begin{CD} \M\times \M @<(a,c)<< \M^{(2)} @>b>> \M\end{CD}\end{equation}
where the morphisms $a,b,c$ take a short exact sequence in $\A$ to its constituent objects, as in the following diagram.
\[\xymatrix@C=.5em{ &0\to A\to B\to C\to 0\ar[dr]_b\ar[dl]^{(a,c)} \\
(A,C)
&&B }\]
It is fairly easy to see that the morphism $(a,c)$ is of finite type, but not representable, whereas $b$ is representable but only locally of finite type. Moreover

\begin{itemize}
\item[(i)] The fibre of $(a,c)$ over $(A,C)\in \M\times \M$ is the quotient stack
\[\big[ \Ext^1_X(C,A)/\Hom_X(C,A)\big].\]
\item[(ii)] The fibre of $b$ over $B\in \M$ is the Quot scheme $\Quot_X(B)$.
\end{itemize}

The idea now is to apply a suitable `cohomology theory' to  our stacks and use the  correspondence \eqref{diagram} to obtain a product operation
\[m\colon H^*(\M)\otimes H^*(\M)\lra H^*(\M).\]
The crucial associativity property  follows from the existence of  certain Cartesian squares involving stacks of 
two-step filtrations. See \cite[Section 4]{bridgeland_intro} for an explanation of this.

By a `cohomology theory' here, we simply mean a rule that assigns a vector space to each stack in such a way that 
\begin{itemize}
\item[(a)] For every morphism of stacks $f\colon X\to Y$, there should exist functorial maps
\[ f^*\colon H^*(Y)\to H^*(X), \quad f_*\colon H^*(X)\to H^*(Y),\]
when $f$ is of finite type or representable respectively,
and satisfying base-change around all suitable 2-Cartesian squares.
%such that for any   Cartesian square 
%\[\begin{CD} U &@>d>>& V\\
% @VcVV && @VVaV \\
%S& @>b>> &T\end{CD}\] with $a$ of finite type and $b$ representable one has  \[a^*\circ
%b_* = d_*\circ c^*\colon H^*(S)\lra H^*(V).\]
\smallskip
\item[(b)] Given two stacks $X$ and $Y$, there should exist functorial K{\"u}nneth maps \[ H^*(X)\tensor H^*(Y) \to H^*(X\times Y).\]
\end{itemize}

We shall see examples of such `cohomology theories' below. Note that the maps in the diagram \eqref{diagram} will not usually be smooth, which makes applying familiar cohomology theories such as  singular cohomology problematic. It seems likely that hidden smoothness results in  derived algebraic geometry will be important in future developments.

\subsection{Grothendieck groups}
The Grothendieck group  $K({\Var}/\C)$ is defined to be 
the free abelian group on the set of isomorphism classes of complex varieties, modulo the scissor relations
\[[X]\sim [Y]+[U],\]
whenever $Y\subset X$ is a closed subvariety and $U=X\setminus Y$. Cartesian product of varieties gives $K(\Var/\C)$ the structure of a commutative ring:\[[X]\cdot [Y]=[X\times Y].\] %By a motivic invariant of complex varieties we mean a ring homomorphism
%\[\Upsilon:K(\Var/\C)\to R.\]

One can of course define Grothendieck rings of complex schemes in the same way. However if one allows arbitrary schemes over $\C$,  an Eilenberg swindle argument using the decomposition
\[\Z\times\Spec(\C)\isom \big(\Z\times\Spec(\C)\big)\bigsqcup \Spec\C\]
will force the ring to be trivial.
 On the other hand, if  one restricts to schemes of finite type over $\C$, the result will be isomorphic to $K(\Var/\C)$, because  any such scheme has a stratification by varieties. 

One can similarly consider relative Grothendieck groups of schemes. Thus given a base scheme $S$ over $\C$ we define $K(\Var/S)$ to be the free abelian group on the set of isomorphism classes of $S$-schemes $f\colon X\to S$, where $X$ is assumed to be of finite type over $\C$,  modulo relations
\[[X\lRa{f}S]\sim[Y\lRa{f|_Y} S]+[U\lRa{f|_U} S],\]
for $Y\subset X$  a closed subscheme  and $U=X\setminus Y$.
Fibre product over $S$ gives a ring structure as before.
Given a map of schemes $\phi\colon S\to T$ there is a group homomorphism
\[\phi_*\colon K(\Var/S) \to K(\Var/T), \quad 
[f\colon X\to S] \mapsto [\phi\circ f\colon X\to T].\]
If the map $\phi$ is of finite type we also get a ring homomorphism
\[\phi^*\colon K(\Var/T) \to K(\Var/S),\quad [g\colon Y\to T]\mapsto [g\times_T S\colon Y\times_T S\to S].\]
There is an obvious K{\"u}nneth type   map
\[[f\colon X\to S]\tensor [g\colon Y\to T] \mapsto [f\times g\colon X\times Y\to S\times T].\]
Together these maps satisfy the basic properties of a `cohomology theory' referred to in Section \ref{general} (although `homology theory' would  perhaps be a more appropriate term in this context).

\subsection{Motivic Hall algebra}

The motivic Hall algebra is defined by taking the `cohomology theory' which assigns to a stack  $S$ the relative Grothendieck ring of stacks over $S$. From now on, all stacks will be assumed to  locally of finite type over $\C$  with affine diagonal.

Given a stack $S$ we define the relative Grothendieck group $K(\St/S)$ to be the free abelian group on the set of isomorphism classes of  $S$-stacks $f\colon X\to S$, where $X$ is assumed to be of finite type over $\C$,   modulo relations
\[[X\lRa{f}S]\sim[Y\lRa{f|_Y} S]+[U\lRa{f|_U} S],\]
for $Y\subset X$  a closed substack  and $U=X\setminus Y$. These relative Grothendieck groups have functorial properties exactly as in the last section.

The motivic Hall algebra is defined to be the relative Grothendieck group
  \[\Hall_{\mot}(\A):=K(\St/\M),\]  with product  defined by the correspondence \eqref{diagram}. Explicitly we have
\[ [Y_1\lRa{f_1}\M] * [Y_2\lRa{f_2} \M] = [Z\lRa{b\circ h}\M], \]
where  $h$ is defined by the Cartesian
square
\begin{equation}
\label{jo}\begin{CD} 
Z & @>h>> &\M^{(2)} &@>b>> \M\\
 @VVV  &&@VV(a,c)V \\
Y_1\times Y_2 &@>f_1\times f_2>> &\M\times\M\end{CD}\end{equation}
Thus, to a first approximation,  an element of the Hall algebra is a family of objects of $\A$ over some base stack $Y$, and the Hall product of two such families is given by taking their universal extension.

One remaining problem is how to define a larger Hall algebra $\hHall_\mot(\A)$  analogous to the algebra $\hHall_{\fin}(\A)$ in the finitary case. This is important because one would like to consider stacks $f\colon \X\to \M$ which are not of finite type, such as the open substack of semistable objects with respect to some stability condition. As explained above, we cannot simply drop the finite type condition since this will lead to the trivial algebra.

The usual solution is rather messy and context-dependent (see e.g. \cite[Sections 5.2--5.3]{bridgeland_curve-counting}) ,  and we do not explain it here: the basic idea is to consider the decomposition $\M=\bigsqcup_{\alpha} \M_\alpha$ according to Chern character, and impose the condition that each $f^{-1}(\M_\alpha)$ is of finite type, together with restrictions on  which of the $f^{-1}(\M_\alpha)$ are allowed to be non-empty.

\subsection{Motivic quotient identity}
\label{motquot}
We now give a rough example of a motivic Hall algebra identity, and explain the sort of reasoning that is required to prove it. 
We take $\A=\Coh(X)$ to be the category of coherent sheaves on a complex projective variety $X$, and look for a version of the identity of Lemma \ref{idid} in the case that $P=\O_X$.

Introduce a stack $\M^\O$ parameterizing sheaves $E\in \Coh(X)$ equipped with a section $ \O_X\to E$. Note that the Hilbert scheme is an open substack
\[\Hilb\subset \M^\O\]
 corresponding to surjective sections.
The analogue of the element $\delta^P$ is  the obvious morphism
$f\colon \M^\O\to \M$
 forgetting the section. The analogue of  $\Quot^P$ is  the induced map $f\colon \Hilb\to \M$. Finally, the analogue of the element $1_\A$ is the identity map $\M\to\M$. 
 
The following result should be taken with a pinch of salt. In particular, we work in an unspecified completion $\hHall_\mot(\A)$.  Rigorous results of a similar kind can be found in \cite[Section 6]{bridgeland_curve-counting}.

\begin{theorem}
 There is an identity 
\[[\M^\O\lRa{f}\M]=[\Hilb\lRa{f} \M] * [\M\lRa{\id}\M],\]
in some suitable completion $\hHall_\mot(\A)$.%where  $f$ and $g$ are the obvious forgetful maps.
\end{theorem}
\smallskip

\noindent {\bf Sketch proof}. The product on the RHS is defined by the Cartesian square
\[\begin{CD}
\mathcal{T} & @>h>> &\M^{(2)} &@>b>> \M\\
 @VVV  &&@VV(a_1,a_2)V \\
\Hilb\times \M &@>f\times \id>> &\M\times\M\end{CD}\]%\smallskip
The points of  the stack $\cT$ over a scheme $S$ are therefore diagrams
\[\xymatrix@C=1.5em{  &\OO_{S\times X} \ar@{.>}[dr]^{\delta}\ar[d]_{\gamma}
\\
0 \ar[r] &A \ar[r]_{\alpha}& B  \ar[r]_\beta &C \ar[r] &0 }
\]
of $S$-flat sheaves on $S\times X$, with $\gamma$ surjective. 
Sending such a diagram to the map $\delta$ defines a morphism of stacks
\[\phi\colon \cT\to \M^\O\]
commuting with the required maps to $\M$.
This map $\phi$   is not an isomorphism of stacks, but it does induce an equivalence on $\C$-valued points, because if $S=\Spec(\C)$, every map $\delta$ factors uniquely via its image: this is the same argument we used in the finitary case.  It follows from this that we can stratify the stack $\M^\O$ by locally-closed substacks such that $\phi$ is an isomorphism over each piece.
This then gives the required identity 
\[[\cT\lRa{b\circ h} \M] = [\M^\O\lRa{f} \M]\] 
in the Grothendieck group $K(\St/\M)$. \qed
 \medskip

{\bf Notes.} The Hall product seems to have been first discovered by Steinitz \cite{steinitz} in 1901 and rediscovered by P. Hall \cite{hall} in 1959. In both cases the category $\A$ was the category of finite abelian $p$-groups. The next  step was taken by Ringel \cite{ringel} who constructed positive parts of quantized enveloping algebras of simple Lie algebras, using Hall algebras of categories of quiver representations over finite fields. %A characteristic zero version of this involving constructible functions was obtained by Schofield and Reidtmann \cite{riedtmann}.
Lusztig \cite{lusztig,lusztig2} used Hall algebras of constructible functions in characteristic zero  to prove his famous results on canonical bases of quantized enveloping algebras.  Schiffmann's lecture notes  \cite{schiffmann_notes2, schiffmann_notes}  cover these developments and much more.
Motivic Hall algebras as described above were first introduced by Joyce \cite{joyce4,joyce6,joyce7}, see also  To{\"e}n \cite[Section 3.3.3]{toen}, and  featured prominently in the work of Kontsevich and Soibelman \cite{kontsevich_soibelman1}. The survey article \cite{bridgeland_intro} covers the basics of this theory.

%%%%%%%%%%%%%%%%%%%%%%%%%%%%%%%%%%%%%%%%%%%%%%%

%%%%%%%%%%%%%%%%%%%%%%%%%%%%%%%%%%%%%%%%%%%%%%%

%%%%%%%%%%%%%%%%%%%%%%%%%%%%%%%%%%%%%%%%%%%%%%%

%%%%%%%%%%%%%%%%%%%%%%%%%%%%%%%%%%%%%%%%%%%%%%%

\section{Integration map}

We have seen in the last section an example of how a basic categorical truth can be translated into an algebraic identity in the Hall algebra, and we will see other important examples below. These identities, while rather aesthetically pleasing, are not usually particularly useful in and of themselves, because the motivic Hall algebra is such a huge and mysterious ring. What makes the theory powerful and applicable is the existence, in certain cases, of  ring homomorphisms from the Hall algebra to much more concrete skew-polynomial rings. These homomorphisms go under the name of `integration maps', since they involve integrating a cohomology class over the moduli space. %Applying this homomorphism to our motivic identites yields identites involving motivic invariants of the relevant moduli spaces.

\subsection{The virtual Poincar{\' e} invariant}
We start by stating the basic properties of the virtual Poincar{\'e} invariant  constructed by Joyce \cite[Sections 4.1--4.2]{joyce5}.
This is an algebra homomorphism
\[\chi_t\colon K(\St/\C)\to \Q(t),\]
uniquely defined by the following two properties:

\begin{itemize}
\item[(i)] If $V$ is a smooth, projective variety then
\[\chi_t(V)=\sum \dim_\C H^i(V^{\an},\C)\cdot t^i\in \Z[t].\]

\item[(ii)] If $V$ is a variety with an action of $\GL(n)$ then
\[\chi_t([V/\GL(n)])=\chi_t(V)/\chi_t(\GL(n)).\]
\end{itemize}

The existence of a virtual Poincar{\'e} polynomial for finite-type schemes over $\C$ follows from the existence of Deligne's mixed Hodge structure on the cohomology groups (see for example \cite{hausel}). A different proof relying on weak factorization can be given using the presentation of the Grothendieck group due to  Bittner \cite{bittner}. The extension to stacks follows from Kresch's result \cite{kresch} that any finite type stack over $\C$ with affine stabilizers has a stratification by  global quotient stacks $[V/\GL(n)]$.

\begin{remarks}
\label{rems}
\begin{itemize}
\item[(a)]
If $V$ is a variety then \[\lim_{t\to -1} \,\chi_t(V)=e(V)\in \Z,\]
 but when $V$ is a stack this limit need not exist, since
 \begin{equation*}
\chi_t(\GL(n))=t^{n(n-1)}\cdot  (t^2-1) (t^4-1) \cdots (t^{2n}-1).\end{equation*}
Often in the theory we shall describe one can construct invariants which are rational functions in $t$. It is then an important and subtle question to determine the  behaviour of these invariants as $t\to -1$. This relates to the question of whether the corresponding elements of the motivic Hall algebra can be represented by varieties rather than stacks.
\smallskip
\item[(b)]
If a variety $V$ is defined over $\Z$, and is cellular in the sense that it has  a stratification by affine spaces, then
\[|V(\F_q)|=\chi_t(V)|_{t=\sqrt{q}},\]
just because both sides are motivic and agree on $\mathbb{A}^k$.  In fact this equality holds whenever $|V(\F_q)|$ is a polynomial in $q$ \cite[Appendix]{hausel}. Thus, setting $q=t^2$, one can expect to compare point counts over $\F_q$ in the finitary world with Poincar{\'e} invariants in the motivic world.
\end{itemize}
\end{remarks}

\subsection{Grothendieck group and charge lattice}
Let $\A$ be an abelian category. From now on we shall assume that
  $\A$
 is linear over a field $k$, and  $\Ext$-finite, in the sense that for all objects $A,B\in \A$ 
\[\dim_k \bigoplus_{i\in \Z} \Ext^i_\A(A,B)<\infty.\]
The most important invariant of such a category is the Euler form
\[\chi(-,-)\colon K_0(\A)\times K_0(\A)\to \Z,\]
defined by the alternating sum
\[\chi(E,F)=\sum_{i\in \Z} (-1)^i \dim_k \Ext^i(E,F).\]

It is often convenient to fix a  group homomorphism \[\ch\colon K_0(\A)\to N\]
to a free abelian group $N$ of finite rank. 
We refer to $N$ as the charge lattice, and $\ch$ as the character map.
We shall always assume that this data satisfies the following two properties:

\begin{itemize}
\item[(i)]The Euler form descends to a bilinear form  $\chi(-,-)\colon N\times N\to \Z$.

\item[(ii)] The character $\ch(E)$ is locally constant in families.%Indeed, given $ [\S\lRa{f} \M]$ we get a finite decomposition
%\[[\S\lRa{f} \M]= \sum_{\alpha\in N} [\S_\alpha\lRa{f|_{\S_\alpha}} \M_\alpha],\]
%where $S_\alpha:=\S\times_\M \M_\alpha \subset \S$ is an open and closed substack.
\end{itemize}
Note that there is then a decomposition \[\M=\bigsqcup_{\alpha\in N} \M_\alpha,\]
into open and closed substacks, and this induces a grading 
\[\Hall_\mot(\A)=\bigoplus_{\alpha\in N} K(\St/\M_\alpha).\]

\begin{examples}
\begin{itemize}
\item[(a)]When $\A=\operatorname{Rep}(Q)$ is the category of finite-dimensional representations of a quiver $Q$, we can take the dimension vector
\[d\colon K_0(\A)\to \Z^{Q_0}.\]

\item[(b)]If $X$ is a smooth complex projective variety we can take\[\ch\colon K_0(\A)\to N=\im(\ch)\subset H^*(X,\Q),\]
to be the  Chern character. The Riemann-Roch theorem shows that the Euler form descends to $N$.
\end{itemize}
\end{examples}
 %such that for all $E,F\in \A$
%\[(\ch(E),\ch(F))=\chi(E,F) .\]
 
\subsection{Quantum torus}
Given a lattice $N\isom \Z^{\oplus n}$  equipped with an integral  bilinear form $(-,-)$, we define a non-commutative  algebra  over the field $\C(t)$ by the rule
\[\C_t[N]=\bigoplus_{\alpha\in N} \C(t)\cdot x^\alpha ,\qquad  x^\alpha * x^\gamma = t^{-(\gamma,\alpha)} \cdot x^{\alpha+\gamma}.\]
This ring is called the quantum torus algebra for the form $(-,-)$. It is a non-commutative deformation of  the group ring $\C[N]$, 
which can be identified with the co-ordinate ring of the algebraic torus
\[\T=\Hom_\Z(N,\C^*)\isom (\C^*)^n.\]
Choosing a basis $(e_1,\cdots,e_n)$ for the group $ N$ gives an identification
\[\C[N]=\C[x_1^{\pm 1},\cdots, x_n^{\pm 1}].\]
The basis elements $(e_1,\cdots, e_n)$ span a positive cone $N_+\subset N$, 
%\[N_+=\big\{\sum_{i=1}^n \lambda_i e_i :\lambda_i\geq 0\big\}\subset N,\]
and we often need the associated completion
\[\C[[N_+]]\isom \C[[x_1,\cdots, x_n]].\]
We define the completed  quantum torus algebra $\C_q[[N_+]]$ in  the same way.

\subsection
{Integration map: hereditary case}

The existence of integration maps is completely elementary when the category $\A$ is hereditary, that is when \[\Ext^i_\A(M,N)=0, \quad i>1.\]
We first consider the case of finitary Hall algebras, and hence assume that $\A$ satisfies the  finiteness conditions of Section \ref{start}. The following result was first proved by Reineke \cite[Lemma 6.1]{reineke2} in the case of representations of quivers. 

\begin{lemma}
\label{inty}
When $\A$ is hereditary there is an algebra homomorphism
\[\I\colon \Hall_\fin(\A)\to \C_t[N]|_{t=\sqrt{q}},\qquad \I(f)=\sum_{E\in \A} \frac{f(E)}{|\Aut(E)|}\cdot x^{\ch(E)},\]
whose codomain is the quantum torus for  the form $2\chi(-,-)$, specialised at $t=\sqrt{q}$.
\end{lemma}

\begin{proof}
Recall the elements $\kappa_E=|\Aut(E)|\cdot\delta_E$, and the identity
\[\kappa_A * \kappa_C= \sum_{B} \frac{|\Ext^1(C,A)_B|}{|\Hom(C,A)|}\cdot\kappa_B\]
of Lemma \ref{ried}. Since $\I(\kappa_E)=x^{\ch(E)}$,
 the result follows immediately from the identity
\begin{equation}
\label{star}\dim_k \Ext^1(C,A)-\dim_k \Hom(C,A)=-\chi(C,A),\end{equation}
which is implied by the hereditary assumption.
\end{proof}

Similar results hold in the motivic case. For example Joyce proved \cite[Theorem 6.1]{joyce4} that when $\A=\operatorname{Rep}(Q)$ is the category of representations of a quiver without relations, or when $\A
=\Coh(X)$ with $X$ a curve, there is an algebra map
\[\I\colon \Hall_\mot(\A) \to \C_{q}[N],\quad \I\big([S \to \M_\alpha]\big)=\chi_t(S)\cdot x^\alpha,\]
to the quantum torus algebra for the form $2\chi(-,-)$.
The basic reason is as for the previous result: the identity \eqref{star} implies that
the fibres of the map $(a,c)\colon \M^{(2)} \to \M\times \M$ in the crucial diagram \eqref{jo} 
 have Poincar{\' e} invariant $t^{-2\chi(\gamma,\alpha)}$ over points in the substack $\M_\alpha\times \M_\gamma$.

\begin{remark}
\label{po}
In the hereditary case it is often more convenient to skew-symmetrise the Euler form by writing
\[\<\alpha,\beta\>=\chi(\alpha,\beta)-\chi(\beta,\alpha).\]
Twisting the integration map by defining
\[\I(f)=\sum_{E\in \A} t^{\chi(E,E)}\cdot \frac{f(E)}{|\Aut(E)|}\cdot x^{\ch(E)}\]
then gives a ring homomorphism to the quantum torus algebra defined by the form $\<-,-\>$. %Morally this somehow corresponds to doubling the category to give a  CY$_3$ category.
In the finitary case, when $\A=\operatorname{Vect}_k$ is the category of vector spaces,  the resulting  map coincides  with that of Section \ref{vect}.
\end{remark}

\subsection{Integration map: CY$_3$ case}
Suppose that $\A=\Coh(X)$ is the category of coherent sheaves on a complex projective Calabi-Yau threefold. Note that the Euler form is skew-symmetric in this case.
Kontsevich and Soibelman \cite[Section 6]{kontsevich_soibelman1} construct  an algebra map \begin{equation}
\label{gl}\I\colon \Hall_\mot(\A)\to \C_t[N],\end{equation}
whose target is  the quantum torus for the Euler form. In fact, much more generally, Kontsevich and Soibelman define an integration map whose target is a version of the quantum torus based on a ring of motives, but we shall completely ignore such generalizations here.
 The definition of this map involves motivic vanishing cycles,  which are beyond the author's competence to explain. There are also some technical problems, for example the existence of orientation data \cite[Section 5]{kontsevich_soibelman1}. %Morally, one can say that the  result depends on the derived  structure of the stack of coherent sheaves on $X$. %One concrete and useful statement is that 
%when $S$ is a scheme

Joyce developed a less ambitious but completely rigorous framework which is sufficient for applications 
 to classical DT invariants. This was repackaged in \cite{bridgeland_intro} in terms of a morphism of Poisson algebras, which can be thought of as the semi-classical limit of Kontsevich and Soibelman's map. In fact, there are two versions of the story, depending on a choice of sign $\epsilon\in \{\pm 1\}$. The sign $+1$ leads to naive DT invariants, whereas $-1$ gives genuine DT invariants. 

We first introduce the semi-classical   limit of the algebra $\C_t[N]$  at $t=\epsilon$: this is a commutative Poisson algebra \[\C_\epsilon[N]=\bigoplus_\gamma \C\cdot x^\gamma\] with product and bracket given by
\[x^\alpha\cdot x^\gamma =\lim_{t\to \epsilon } \big({x^\alpha * x^\gamma}\big)=\epsilon^{\<\alpha,\gamma\>} \cdot x^{\alpha+\gamma},\]
\[\{x^\alpha,x^\gamma\} =\lim_{t\to \epsilon } \Big(\frac{x^\alpha * x^\gamma-x^\gamma* x^\alpha}{t^2-1} \Big)= \<\alpha,\gamma\>\cdot x^{\alpha}\cdot x^{\gamma}.\]

The next step is to introduce a similar semi-classical limit of  the motivic Hall algebra \cite[Section 5] {bridgeland_intro}. One first defines a subalgebra of `regular' elements
\begin{equation}\label{reg}\Hall_{\reg}(\A)\subset  \Hall_\mot(\A).\end{equation} To a first approximation it is the subspace spanned  by the symbols $[X\to \M]$ in which  $X$ is a scheme, rather than a stack. The limit as $t\to\epsilon$ can then be taken exactly as above to give a commutative Poisson algebra called the semi-classical Hall algebra $\Hall_{\rm sc}(\A)$.

One can now define a morphism of Poisson algebras
\begin{equation}
\label{saz} I_\epsilon \colon \Hall_{{\rm sc}}(\A) \to \C_\epsilon[N_+]\end{equation}
 by the formula 
\[\I \big([S \lRa{f} \M_\alpha]\big)=\begin{cases} e(S) \cdot x^\alpha &\text{if }\epsilon=+1,\\ e(S;f^*(\nu)) \cdot x^\alpha & \text{if }\epsilon= -1,\end{cases}\]
 where $\nu\colon \M\to \Z$ is the Behrend function appearing in the definition of DT invariants. 

When $\epsilon=1$, the fact that $\I_\epsilon$ is a Poisson map  just requires  the identity
%\[( \aext^1(C,A) -\ahom(C,A))-( \aext^1(A,C) -  \ahom(A,C))=\chi(A,C),\]
\[\chi(A,C)=\big( \aext^1(C,A) -\ahom(C,A)\big)\]\[-\big( \aext^1(A,C) -  \ahom(A,C)\big),\]
which follows from the CY$_3$ assumption. In the case $\epsilon=-1$, one also needs some identities involving the Behrend function  proved by Joyce and Song \cite[Theorem 5.11]{joyce_song}.

\smallskip

{\bf Notes.} The first occurrence of an integration map is perhaps in Reineke's paper \cite{reineke2}.  This was generalised to the setting of motivic Hall algebras by Joyce  \cite[Section 6]{joyce4}. Joyce also constructed an integration map in the CY$_3$ case that is a map of Lie algebras. It was Kontsevich and Soibelman's remarkable insight \cite{kontsevich_soibelman1} that incorporating vanishing cycles could lead to an integration map which is a homomorphism of algebras. Following this, Joyce and Song \cite{joyce_song} were able to  incorporate  the Behrend function into Joyce's Lie algebra map. The interpretation  in terms of semi-classical limits and Poisson algebras can be found in  \cite{bridgeland_intro}.

%%%%%%%%%%%%%%%%%%%%%%%%%%%%%%%%%%%%%%%%%%%%

%%%%%%%%%%%%%%%%%%%%%%%%%%%%%%%%%%%%%%%%%%%%

%%%%%%%%%%%%%%%%%%%%%%%%%%%%%%%%%%%%%%%%%%%%

\section{Generalized DT invariants}

One of the most important aspects of the work of Joyce, and of Kontsevich and Soibelman, is the generalization of Donaldson-Thomas invariants associated to moduli spaces of stable sheaves developed in \cite{thomas} to the case when there exist strictly semistable objects. The resulting invariants satisfy a wall-crossing formula which controls their behaviour under change of stability condition. Here we give a brief outline of these constructions and explain the simplest examples.

\subsection{The problem}
Let $X$ be a smooth projective Calabi-Yau threefold, and set $\A=\Coh(X)$.
Fix a polarization  of $X$ and a class $\alpha\in N$,  and consider the stack  
\[\M^{ss}(\alpha)=\big\{E\in \Coh(X): E \text{ is  Gieseker semistable   with }\ch(E)=\alpha\big\}.\]
We also consider the unions of these stacks given by sheaves of a fixed slope
\[\M^{ss}(\mu)=\big\{E\in \Coh(X): E \text{ is  Gieseker semistable   of slope }\mu(E)=\mu\big\}.\]
Note that we consider the zero object to be semistable of all slopes $\mu$.

In the case when  $\alpha$ is primitive,  and the polarization is general, the stack  $\M^\ss(\alpha)$
 is a $\C^*$-gerbe over its coarse moduli space $M^{\ss}(\alpha)$, and  we  can set \[\ET(\alpha)=e(M^{\ss}(\alpha))\in \Z.\]
Genuine DT invariants, as defined by Thomas \cite{thomas},
 are defined using virtual cycles, or by a weighted Euler characteristic as before. The problem is then to generalize these invariants to arbitrary classes $\alpha\in N$. It turns out that  even if one is only interested in the invariants $\DT(\alpha)$ for  primitive $\alpha$,  to understand the behaviour of these invariants as the polarization $\ell$ is varied, one in fact needs to treat all  $\alpha$ simultaneously.

For a general class $\alpha\in N$, the moduli stack  $\M^{ss}(\alpha)$ at least has a well-defined Poincar{\'e} function \[\qET(\alpha)=\chi_t(\M^{ss}(\alpha))\in \Q(t),\] 
which we can view as a kind of naive quantum DT invariant. When $\alpha$ is primitive, the fact that $\M^{\ss}(\alpha)$ is a $\C^*$-gerbe over the coarse moduli space $M^{\ss}(\alpha)$, together with Remark \ref{rems}(a), implies that
\[\ET(\alpha)=\lim_{t\to 1} \,(t^2-1)\cdot \qET(\alpha) \in \Z.\]
In general however, $\qET(\alpha)$ has higher-order poles at $t=1$, so it is not immediately clear how to define $\ET(\alpha)$.

\subsection{The solutions}
Joyce \cite{joyce7} worked out how to define  invariants \[\ET(\alpha)\in \Q\] for arbitrary classes $\alpha\in N$, and showed that they satisfy a wall-crossing formula as the polarization is varied. Incorporating the Behrend function into Joyce's framework  leads to generalized DT invariants $\DT(\alpha)\in\Q$ also satisfying a wall-crossing formula \cite{joyce_song}. These results rely on a very deep result \cite[Theorem 8.7]{joyce6} known as the no-poles theorem, which implies that the element 
\[[\C^*]\cdot \log \big([\M^{\ss}(\mu)\subset \M]\big) \in \hHall_\mot(\A),\]
obtained by applying the Taylor expansion of $\log(1+x)$, lies in the subalgebra $\hHall_{\reg}(\A) \subset \hHall_\mot(\A)$ discussed above. (Recall that $\M^{\ss}(\mu)$ includes a component corresponding to the zero object). Applying the Poisson integration  map \eqref{saz}  to this element then leads to a generating function 
\[\DT_\mu =-\lim_{q\to 1} (q-1)\cdot   \log\, \qDT_\mu\in \C[[N_+]] \] whose coefficients are the required invariants.

In  a different approach, Kontsevich and Soibelman \cite{kontsevich_soibelman1}  use motivic vanishing cycles to define  genuine quantum DT invariants, which  are again rational functions \[\qDT(\alpha)\in\Q(t).\] In fact they do much more: they define motivic invariants lying in the ring $K(\St/\C)$, but we shall suppress this extra level of complexity here. Note however that these results rely on the currently unproven existence of orientation data. In terms of the  map \eqref{gl}, one first considers \[\qDT_\mu=\I\big([\M^{ss}(\mu)\subset \M]\big)\in \C_t[[N_+]],\]
and sets $\qDT(\alpha)$ to be the coefficient of $x^\alpha$. 
Kontsevich and Soibelman \cite[Section 7]{kontsevich_soibelman1} also formulated a conjecture, closely related to Joyce's no-poles theorem, which states that
\[\DT_\mu =-(t^2-1)\cdot   \log\, \qDT_\mu\in \C_t[[N_+]] \]
should be regular at $t=\pm 1$. Assuming this, one can recover Joyce's invariants by setting $t=1$.

Conjugation by the  quantum DT generating function give rise to an automorphism of the quantum torus algebra
\[\operatorname{q-\S_\mu}=\operatorname{Ad} _{\qDT(\mu)}(-)\in \Aut \C_t[[N_+]].\]
The no-poles conjecture implies that this automorphism has a well-defined limit at $t=1$ which is the 
Poisson automorphism
\[\S_\mu=\exp\big\{\small{-}\rm{DT}_\mu,-\big\}\in \Aut \C[[N_+]].\]
Geometrically, this can be thought of as the action of the time 1 flow of the Hamiltonian vector field generated by the DT generating function $\DT_\mu$.

\subsection{Example: a single spherical bundle}

Suppose we are in the simplest possible situation when there is a unique stable bundle $E$ of slope $\mu$, which is moreover rigid, i.e. satisfies
$\Ext^1_X(S,S)=0.$
Serre duality implies that $S$ is in fact spherical.
  The category of semistable sheaves of slope $\mu$ is then equivalent to the category of finite-dimensional vector spaces, so \[\M^{ss}(\mu)=\{E^{\oplus n}:n\geq 0\}\isom \bigsqcup_{n\geq 0} \operatorname{BGL}(n,\C).\]
The Kontsevich-Soibelman integration map for this category is  closely related to the ring homomorphism $I$  considered in Section \ref{vect}.
Setting $\alpha=\ch(E)\in N$ we can compute

\begin{itemize}
\item[(a)] The quantum DT generating function is
\[\qDT_\mu =\sum_{n\geq 0} \frac{t^n\cdot x^{n\alpha}}{(t^{2n}-1) \cdots (t^2-1)}\in \C_t[[N_+]].\]
We recognise the quantum dilogarithm $\E_{t^2}(x^\alpha)$.

\item[(b)]
The classical DT generating function is 
\[\DT_\mu=-\lim_{t\to -1} (t^2-1)\cdot  \log  \E_{t^2}(x^\alpha) =\sum_{n\geq 1} \frac{x^{n\alpha}}{n^2},\]
and we conclude that $\DT(n\alpha)=1/n^2$.\smallskip

\item[(c)]
The Poisson automorphism $\S_\mu\in \Aut \C[[N_+]]$ is
\begin{equation}\label{clus}\S_\mu(x^\beta)= \exp \bigg\{\small{-}\sum_{n\geq 1} \frac{x^{n\alpha}}{n^2}, -\bigg\} (x^\beta)=x^\beta \cdot (1- x^\alpha)^{\<\alpha,\beta\>}.\end{equation}
\end{itemize}

The right-hand side of this identity \eqref{clus} should be expanded as a power series to give an element of $\C[[N_+]]$. However we can also view $\S_\mu$ as defining a birational  automorphism of the Poisson torus $\T$. Viewed this way, it is the basic example of a cluster transformation.

\subsection{Stability conditions}
\label{stab}
We shall now move on to discussing the behaviour of DT invariants under changes of stability parameters. Although the results apply perfectly well to the context of Gieseker stability,  the picture is perhaps clearer for  stability conditions in the sense of \cite{bridgeland_stability} which we now review.
We fix an abelian category  $\A$ throughout.

\begin{definition}
A stability condition on $\A$ is a map of  groups
$Z\colon K_0(\A)\to \C$
 such that
\[0\neq E\in \A \implies Z(E)\in \bar{\IH},\]
where $\bar{\IH}=\IH\cup\R_{<0}$ is the semi-closed upper half-plane.
\end{definition}

%Here $\bar{\IH}$ is the semi-closed upper half-plane
%\[\bar{\IH}=\{z\in \C:z\in \R_{>0}\cdot \exp(i\pi \phi)\text{ with }0<\phi\leq 1\}.\]

%\vspace{.3cm}
%\begin{figure}
%\begin{center}
%\begin{tikzpicture}[scale=0.4]
%\draw (0,0)--(4.9,0);
%\draw[dotted] (5.1,0)--(10,0);
%\draw (5,4) node {$\bar{\IH}$};
%\draw (5,0) circle [radius=0.1];
%\draw[->,red] (5.1,0.1) -- (9,3);
%\draw[->,red] (4.8,0.1) -- (1,2);
%\draw (10.5,3) node {\small $Z(M_1)$};
%\draw (-0.5,2) node {\small $Z(M_2)$};
%\draw[fill] (9.05,3.05) circle [radius=0.05];
%\draw[fill] (.95,2.05) circle [radius=0.05];
%\end{tikzpicture}
%\end{center}
%\end{figure}
%\smallskip

%

%\begin{example}
%Take $\A=\Coh(X)$ with $X$ a curve and $Z(E)=-\deg(E)+i\cdot \rank(E)$.
%\end{example}

The phase of a nonzero object $E\in \A$ is
\[\phi(E)=\frac{1}{\pi}\arg Z(E)\in (0,1],\]
A nonzero object  $E\in \A$  is  said to be $Z$-semistable if
\[0\neq A\subset E\implies \phi(A)\leq \phi(E).\]
We let $\P(\phi)\subset \A$ be the full additive subcategory of $\A$ consisting of the nonzero $Z$-semistable objects of phase $\phi$, together with the zero objects. %It is not hard to see that $\P(\phi)$ is an abelian category.

%\vspace{.3cm}
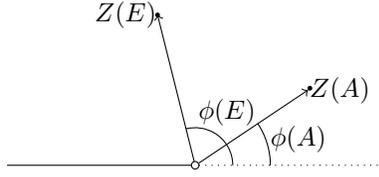
\begin{figure}
\begin{center}
\begin{tikzpicture}[scale=0.5]
\draw (-0.1,0)--(-5,0);
\draw[dotted] (0.1,0)--(5,0);
%\draw (5,4) node {$\bar{\IH}$};
\draw (0,0) circle [radius=0.1];
\draw[->] (0.1,0.06) -- (3,2);
%\draw[->,red] (-0.1,0.05) -- (-4,2);
\draw[->] (-0.025,0.1) -- (-1,4);
\draw (3.85,2) node {\footnotesize $Z(A)$};
%\draw (-4.85,2) node {\footnotesize$Z(B)$};
\draw (-1.85,4) node {\footnotesize $Z(E)$};
\draw[fill] (3.05,2.05) circle [radius=0.05];
%\draw[fill] (-4,2) circle [radius=0.05];
\draw[fill] (-1,4) circle [radius=0.05];
\draw(2,0) arc [radius=2, start angle=0, end angle= 33.7];
\draw(1,0) arc [radius=1, start angle=0, end angle= 104];
\draw (2.75,.7) node {\footnotesize $\phi(A)$};
\draw (.85,1.4) node {\footnotesize $\phi(E)$};
\end{tikzpicture}
\caption{Central charges and phases}
\end{center}
\end{figure}

We say that a stability condition $Z$  has the Harder-Narasimhan property if every object $E\in \A$ has a  filtration
\[0=E_0\subset E_1\subset \cdots \subset E_n\subset E\]
such that each factor $F_i=E_i/E_{i-1}$ is nonzero and $Z$-semistable and
\[\phi(F_1)> \cdots >\phi(F_n).\]
Existence of such filtrations is a fairly weak condition: for example if $\A$ is of finite length (Artinian and Noetherian) it is automatic. When they exist, Harder-Narasimhan filtrations are necessarily unique, because the usual argument shows that if $F_1,F_2$ are $Z$-semistable then
\[\phi(F_1)>\phi(F_2) \implies \Hom_\A(F_1,F_2)=0,\]
and another standard argument then gives uniqueness.

%\begin{remarks}
%\begin{itemize}
%\item[(i)] If $\A$ has finite length this condition is automatic.
%\item[(ii)] 
%When they exist, Harder-Narasimhan filtrations are necessarily unique, because the usual argument shows that if $F_1,F_2$ are $Z$-semistable then
%\[\phi(F_1)>\phi(F_2) \implies \Hom(F_1,F_2)=0.\]
%\end{itemize}
%\end{remarks}

%It is the existence and uniqueness of  Harder-Narasimhan filtrations which gives rise to the crucial wall-crossing formula in Donaldson-Thomas theory.
%Note that repeated products in $H(\A)$ are given by sums over flags:
%\[(f_1*\cdots *f_n)(M)=\sum_{0=M_0\subset M_1\subset \cdots \subset M_n= M} f_1(M_1/M_0) \cdots f_n(M_n/M_{n-1}).\]

\subsection{Wall-crossing identity}
\label{wc}
Let  us consider the wall-crossing formula in the finitary context. So assume that $\A$ is an abelian category satisfying the strong finiteness conditions of Section \ref{start}. Let us also equip $\A$  with a stability condition $Z$ having the Harder-Narasimhan property. 
Let \[\delta^{\ss}(\phi)\in \hHall_{\fin}(\A)\] be the characteristic function of  the subcategory $\P(\phi)\subset \A$.  We define the element $\delta_\A\in \hHall_{\fin}(\A)$ as in Section \ref{start}. The following crucial result was first proved by Reineke \cite{reineke2}.

\begin{lemma}
\label{god}There is an identity
\[\delta_\A= \prod^{\to}_{\phi\in \R} \delta^{\ss}(\phi)\]
in the Hall algebra $\hHall_{\fin}(\A)$, where the product is taken in descending order of phase.
\end{lemma}

\begin{proof}
To make sense of the infinite product, we first write $\delta^\ss(\phi)=1+\delta^\ss(\phi)_+$ where $\delta^\ss(\phi)_+$ is the characteristic function of the set of nonzero semistable objects of phase $\phi$.
Then we can rewrite the infinite product as an infinite sum
\[\prod^{\to}_{\phi\in \R} \delta^{\ss}(\phi)=\prod^{\to}_{\phi\in \R} (1+\delta_+^{\ss}(\phi))=1+\sum_{k\geq 1} \sum_{\phi_1>\cdots>\phi_k} \delta^\ss_+(\phi_1) * \cdots * \delta^\ss_+(\phi_k).\]
Using the formula \eqref{multiple} for multiple products in the Hall algebra, it is clear that evaluating the right-hand side on any object $M\in \A$ produces a sum over the finitely many filtrations of $M$, each taken with coefficient $0$ or $1$. Moreover, a filtration has coefficient 1 precisely if its factors are $Z$-semistable with descending phase. The identity thus follows from existence and uniqueness of Harder-Narasimhan filtrations.
\end{proof}
 
The left-hand side of the  identity of Lemma \ref{god} is independent of the stability condition $Z$. Thus  given two stability conditions on $\A$ we get a wall-crossing formula
\begin{equation}
\label{wcc}\prod^{\lra}_{\phi\in\R}\delta^{\ss}(\phi,Z_1) = \prod^{\lra}_{\phi\in\R} \delta^{\ss}(\phi,Z_2).\end{equation}
If $\A$ is moreover hereditary we can then apply the integration map of Lemma \ref{inty}  to get an identity in the corresponding completed quantum torus algebra  $\C_t[[N_+]]$.
Considering the  automorphisms of $\C_t[[N_+]]$ given by conjugation of the two sides of \eqref{wcc}, and  taking the limit as $t\to 1$, we also obtain an identity
in the group of  automorphisms of the Poisson algebra  $\C[[N^+]]$. We will work through the simplest non-trivial  example of this in the next subsection.

\subsection{Example: the $A_2$ quiver}
Let $Q$ be the $A_2$ quiver: it has two vertices $1$ and $2$, and a single arrow from $1$ to $2$.
Let $\A$ be the abelian category of finite-dimensional representations of $Q$ over the field $k=\F_q$.  This category has exactly three indecomposable representations, which fit into a short exact sequence \[0\lra S_2\lra E\lra S_1\lra 0.\]
Here $S_1$ and $S_2$ are the simple representations at the vertices $1$ and $2$ respectively, and $E$ is the unique indecomposable representation of dimension vector $(1,1)$.
We have $N=K_0(\A)=\Z^{\oplus 2}=\Z[S_1]\oplus\Z[S_2]$. As in Remark \ref{po} we consider  the skew-symmetrised Euler form
\[\<(m_1,n_1),(m_2,n_2)\>=m_2 n_1-m_1 n_2.\]
The corresponding quantum torus algebra is
\[\C_t[[N_+]]=\C\<\<x_1,x_2\>\>/(x_2*x_1-t^2\cdot x_1 *x_2),\]
and its semi-classical limit at $t=1$ is the Poisson algebra
 \[\C[[N_+]]= \C[[x_1,x_2]],\quad \{x_1,x_2\}=-x_1\cdot x_2.\]
A stability condition on $\A$ is determined by the pair $(Z(S_1),Z(S_2))$, so the space of all  such stability conditions is $\Stab(\A)\isom \bar{\IH}^2$. There is a single wall
\[\mathcal{W}=\{Z\in \Stab(\A): \Im Z(S_2)/Z(S_1)\in \R_{>0}\},\]
where the object $E$ is strictly semistable.
The complement of this wall consists of two chambers: in one $E$ is strictly stable, in the other it is unstable.

\begin{figure}
\begin{tikzpicture}[scale=0.5]
\draw (0,0)--(4.9,0);
\draw[dotted] (5.1,0)--(10,0);
\draw (5,0) circle [radius=0.1];
\draw[->] (5.1,0.1) -- (8,2);
\draw[->] (4.9,0.1) -- (3,2);
\draw[->,dashed](5.1,0.1)--(6,4);
\draw (9.4,2) node { $\scriptstyle Z(S_1)$};
\draw (1.75,2) node {$\scriptstyle Z(S_2)$};
\draw (4.8,4) node {$\scriptstyle Z(E)$};

\draw[fill] (8.05,2.05) circle [radius=0.05];
\draw[fill] (2.95,2.05) circle [radius=0.05];
\draw[fill] (5.95,4.05) circle [radius=0.05];

\draw[thick](12.5,-1.5)--(12.5,4);
\draw (12.5,5) node {\small $\mathcal{W}$};
\draw (19.9,0)--(15,0);
\draw[dotted] (20.1,0)--(25,0);

\draw (5,-1) node {$E$ unstable};
\draw (20,-1) node {$E$ stable};

\draw (20,0) circle [radius=0.1];
\draw[->] (20.1,0.06) -- (23,2);
\draw[->] (19.9,0.05) -- (16,2);
\draw[->] (19.975,0.1) -- (19,4);
\draw (24.25,2) node { $\scriptstyle Z(S_2)$};
\draw (14.75,2) node { $\scriptstyle Z(S_1)$};
\draw (17.9,4) node { $\scriptstyle Z(E)$};
\draw[fill] (23.05,2.05) circle [radius=0.05];
\draw[fill] (16,2) circle [radius=0.05];
\draw[fill] (19,4) circle [radius=0.05];
\end{tikzpicture}
\caption{Wall-crossing for the A$_2$ quiver.}%: there is just one wall $\mathcal{W}$ corresponding to stability conditions for which $Z(S_1)$ and $Z(S_2)$ are parallel. The simple objects $S_1,S_2$ are always stable, but $E$ is only stable on one side of the wall.} 
\end{figure}
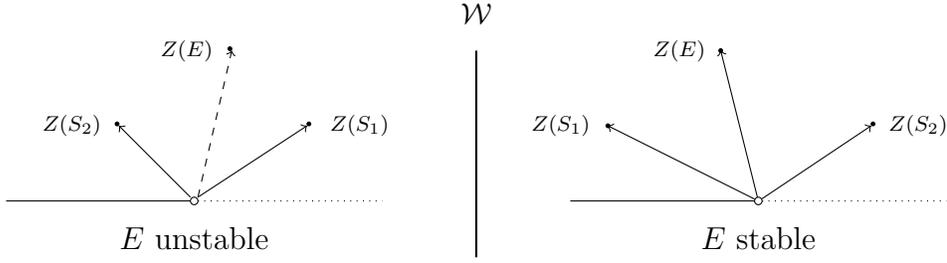

The wall-crossing formula in $\C_t[[N_+]]$ becomes the  identity
\[\E_q(x_2)*\E_q(x_1)=\E_q(x_1)*\E_q(t\cdot x_1 *x_2)*\E_q(x_2),\]
where $q=t^2$ as usual. This is known as the pentagon identity for the quantum dilogarithm: see \cite[Section 1]{keller} for references.
The semi-classical version of the wall-crossing formula is the cluster identity
\[C_{(0,1)} \circ C_{(1,0)} = C_{(1,0)} \circ C_{(1,1)} \circ C_{(0,1)}.\]
\[C_\alpha\colon x^\beta \mapsto  x^\beta\cdot (1+x^\alpha)^{\<\alpha,\beta\>}\in \Aut \C[[x_1,x_2]].\]
It can be viewed in the group of birational automorphisms of $(\C^*)^2$ which preserve the invariant symplectic form.
%\[\omega=\frac{dx_1}{x_1}\wedge \frac{dx_2}{x_2}.\]

\smallskip

{\bf Notes.} The crucial observation that the existence and uniqueness of Harder-Narasimhan filtrations leads to an identity in the Hall algebra is due to Reineke \cite{reineke2}. This idea was taken up by Joyce to give a wall-crossing formula for naive Donaldson-Thomas invariants \cite[Theorem 6.28]{joyce7}. Joyce's formula is combinatorially messy, although perfectly usable \cite{toda2,toda6}. It was Kontsevich and Soibelman \cite{kontsevich_soibelman1} who uncovered the connection with cluster transformations. We recommend Keller's article \cite{keller} for more on the wall-crossing formula in the context of representations of quivers.

%%%%%%%%%%%%%%%%%%%%%%%%%%%%%%%%%%%%%%%%%%%%

%%%%%%%%%%%%%%%%%%%%%%%%%%%%%%%%%%%%%%%%%%%%

%%%%%%%%%%%%%%%%%%%%%%%%%%%%%%%%%%%%%%%%%%%%

\section{Framed invariants and tilting}

It often happens that  the invariants in which one is interested relate to objects of an abelian category equipped with some kind of framing. 
For example, the Hilbert scheme parameterizes  sheaves $E\in \Coh(X)$ equipped with a surjective map \[f\colon \O_X\onto E.\]
One immediate advantage  is that the framing data eliminates all stabilizer groups, so the moduli space is  a scheme, and therefore has a well-defined Euler characteristic. On the other hand it is less obvious how to consider wall-crossing in this framework: what is the stability condition which we should vary?
In fact  
 wall-crossing can often be achieved in this context by varying the t-structure on the derived category $\D^b\Coh(X)$. This has the effect of varying which maps $f$ are considered to be surjective.  %The examples of wall-crossing formulae in the introduction can be viewed as examples of variation of stability for ideal sheaves, viewed as rank 1 objects in the derived category. However a simpler approach is to consider the moduli spaces as parameterizing sheaves supported in dimension $\leq 1$ with a framing. Wall-crossing is then acheived by varying the heart.

\subsection{T-structures and hearts}
We recall the definition of a bounded t-structure.
Let $\D$ be a triangulated category.

\begin{defn}
 A heart $\A\subset\D$ is a full  subcategory such that:
\smallskip

\begin{itemize}
\item[(a)]$\Hom(A[j],B[k])=0$ for all $A,B\in\A$ and
$j>k$.
\smallskip

\item[(b)] for every object $E\in\D$ there is a finite filtration
\[0=E_m\to E_{m+1}\to \cdots \to E_{n-1}\to E_n=E\]
with factors
$F_j=\operatorname{Cone}(E_{j-1}\to E_j) \in \A[-j]$. 
\end{itemize}
\end{defn}

In condition (b) the word filtration really means  a
finite sequence of triangles
\[
\xymatrix@C=.2em{ 0_{\ } \ar@{=}[r] & E_{m-1} \ar[rrrr] &&&& E_{m}
\ar[dll] \ar[rr] \ar[dll] && \cdots \ar[rr] && E_{n-1}
\ar[rrrr] &&&& E_n \ar[dll] \ar@{=}[r] &  E_{\ } \\
&&& F_{m} \ar@{-->}[ull]  &&&&&&&& F_n \ar@{-->}[ull] }
\]
with $F_j\in\A[-j]$.

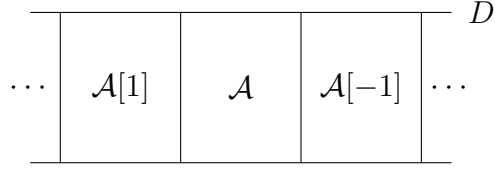
\begin{figure}
\begin{center}
\begin{tikzpicture}[scale=0.4]
\draw (0,4)--(14,4);

\draw (0,9)--(14,9);
\draw (1,4)--(1,9);
\draw (5,4)--(5,9);
\draw (9,4)--(9,9);
\draw (13,4)--(13,9);
\draw (0,6.5) node {$\cdots$};
\draw (14,6.5) node {$\cdots$};
\draw (3,6.5) node {$\A[1]$};
\draw (7,6.5) node {$\A$};
\draw (11,6.5) node {$\A[-1]$};
\draw (15,9) node {$\D$};

\end{tikzpicture}
\caption{The `film-strip' picture of a t-structure.}
\end{center}
\end{figure}

It would be more {standard} to say that $\A\subset \D$ is the heart of a bounded t-structure on $\D$. But any such t-structure is determined by its heart. The basic example is $\A\subset
\D^b(\A)$. In analogy with that case we define \[H^j_\A(E):=F_j[j]\in \A.\]
It follows from the above definition that $\A$ is in fact an abelian category.
The short exact sequences in $\A$ are precisely the  triangles in $\D$ all of whose terms lie in $\A$. Finally, the inclusion functor gives a canonical identification $K_0(\A)\isom K_0(\D)$.

%\begin{figure}
%\begin{center}
%\begin{tikzpicture}[scale=0.4]
%\draw (0,4)--(14,4);
%
%\draw (0,9)--(14,9);
%\draw (1,4)--(1,9);
%\draw (5,4)--(5,9);
%\draw (9,4)--(9,9);
%\draw (13,4)--(13,9);
%\draw (0,6.5) node {$\cdots$};
%\draw (14,6.5) node {$\cdots$};
%\draw (3,6.5) node {$\A[1]$};
%\draw (7,6.5) node {$\A$};
%\draw (11,6.5) node {$\A[-1]$};
%\draw (15,9) node {$\D$};
%
%\end{tikzpicture}
%\end{center}
%\end{figure}

\subsection{Tilting at torsion pairs}
We now explain how to tilt a heart at a torsion pair \cite{tilt}. This is an important method for obtaining new t-structures from old.

\begin{defn}Let $\A$ be an abelian category.
A torsion pair $(\TT,\FF)\subset \A$ is a pair of full subcategories
such that:
%\smallskip

\begin{itemize}
\item[(a)]
$\Hom_\A(T,F)=0$ for $T\in
\TT$ and $F\in\FF$.\smallskip

\item[(b)]
for every object $E\in\A$ there is a  short
exact sequence
\[0\lra T\lra E\lra F\lra 0\] for some pair of objects $T\in\TT$ and
$F\in  \FF$.
\end{itemize}
\end{defn}

% A good example is when $\A=\Coh(X)$ and $\TT$ and $\FF$ consist of torsion and torsion-free sheaves respectively.
Suppose $\A\subset \D$ is a heart, and   $(\TT,\FF)\subset\A$  a torsion pair.
We can define a new heart $\A^\sharp\subset\D$ such that 
an object $E\in \D$ lies in $\A^\sharp\subset \D$ precisely if   \[H_\A^{0}(E)\in \FF,\quad H_\A^1(E)\in \TT, \quad H^i_\A(E)=0\ \text{ otherwise}.\] 
This process is illustrated in Figure \ref{tiltfig}. The heart $\A^\sharp$ is called the right tilt of the heart $\A$ at the torsion pair $(\TT,\FF)$. The left tilt is the subcategory $\A^\sharp[1]$. %\[\A^\sharp=\<\FF[1],\TT\>\subset \D.\] 
%called the tilt of $\A$ at the torsion pair $(\TT,\FF)$.

\begin{figure}
\begin{center}
\begin{tikzpicture}[scale=0.4]
\draw (-6,4)--(19,4);

\draw (-6,9)--(19,9);
\draw (1.25,4)--(1.25,9);
\draw (4.25,4)--(4.25,9);
\draw (9,4)--(9,9);
\draw (12,4)--(12,9);

\draw(-3.75,4)--(-3.75,9);
\draw(16.75,4)--(16.75,9);

\draw (2.75,6.5) node {$\TT$};
\draw (6.5,6.5) node {$\FF$};
\draw (10.5,6.5) node {$\TT[-1]$};
\draw (-1.25,6.5) node {$\FF[1]$};

\draw (14.25,6.5) node {$\FF[-1]$};

\draw (18.25,6.5) node {$\cdots$};
\draw (-4.75,6.5) node {$\cdots$};
%\draw (14,9) node {$\D$};
\draw [decorate, decoration={brace,amplitude=5pt}] (4,9.5)--(11.75,9.5)
node [midway, above=6pt] {$\A^\sharp$};
%\draw [decorate, decoration={brace,amplitude=5pt}] (-3.75,9.5)--(4,9.5)
%node [midway, above=6pt] {$\A[1]$};

\draw [decorate, decoration={brace,amplitude=5pt}] (9.25,3.5)--(1.5, 3.5)
node [midway, below=6pt] {$\A$};
\end{tikzpicture}
\caption{Tilting a heart $\A\subset \D$ at a torsion pair $(\TT,\FF)\subset \A$. \label{tiltfig} }
\end{center}
\end{figure}
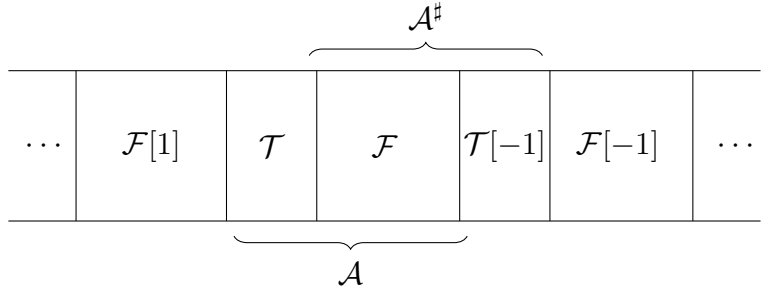

\subsection{Examples of tilts}

Let us consider the right tilt of the standard heart \[\A=\Coh(X)\subset \D^b\Coh(X)\] with respect to the torsion pair
\[\TT=\{E\in \Coh(X):\dim \operatorname{supp}(E)=0\},\]
\[\FF=\{E\in \Coh(X): \Hom_X(\O_x,E)=0 \text{ for all }x\in X\}.\]
Thus $\TT$ consists of zero-dimensional  sheaves, and $\FF$ consists of sheaves with no zero-dimensional torsion.
Note that \[\O_X\in \FF\subset \A^\sharp.\]
We claim that the stable pairs moduli space of Section \ref{poop}  is the analogue of the Hilbert scheme in this tilted context.

\begin{lemma}
The stable pairs moduli space $\PHilb(\beta,n)$ parameterizes quotients of $\O_X$ in the tilted category $\A^\sharp$:
\[\PHilb(\beta,n)=\bigg\{\parbox{13em}{\centering \rm quotients $\O_X\onto E$ in $\A^\sharp$ with $\ch(E)=(0,0,\beta,n)$}\bigg\}.\]
\end{lemma}
%\smallskip

\begin{proof} Given a short exact sequence in the category $\A^\sharp$
\[0\lra J\lra \O_X\lRa{f} E\lra 0,\]
we can take cohomology with respect to the standard heart $\A\subset \D$ to get a long exact sequence  in the category $\A$ 
\begin{equation}
\label{gard}0\to H_\A^0(J) \to \O_X \lRa{f} H_\A^0(E) \to H_\A^1(J)\to 0\to H^1_\A(E)\to0.\end{equation}
 It follows that $E\in \A\cap\A^{\sharp}=\FF$ and $\coker(f)=H^1_\A(J)\in \TT$. This is precisely the condition that $f\colon \O_X\to E$ defines a stable pair.

 For the converse, take a stable pair and embed it in a triangle
 \begin{equation}
 \label{gar}J\lra \O_X\lRa{f} E \lra J[1].\end{equation}
 By definition $E\in \FF\subset \A^\sharp$. 
 The same long exact sequence \eqref{gard} then shows that $J\in \A^\sharp$. It follows that \eqref{gar} defines a short exact sequence in $\A^\sharp$.
 \end{proof}

Tilting also allows to give a precise description of the effect of a threefold flop \[\xymatrix@C=1.4em{ X_+\ar[dr]_{f_+}&&
X_-\ar[dl]^{f_-} \\
&Y }\]
on the derived category. Suppose for simplicity that each map $f_\pm$ contracts a single rational curve $C_\pm$. Introduce subcategories
\[\FF_+=\< \O_{C_+}(-i)\>_{i\geq 1}\subset \Coh(X_+),\quad \FF_-=\< \O_{C_-}(-i)\>_{i\geq 2}\subset \Coh(X_-),\]
where the angular brackets denote extension-closure.
These subcategories turn out to be torsion-free parts of torsion pairs on the categories $\Coh(X_\pm)$ \cite{vdb}.   Moreover, the equivalence $\D(X_+)\isom \D(X_-)$ constructed in \cite{bridgeland_flop}  induces an exact equivalence between the corresponding tilted categories $\Per^{\pm}(X_\pm/Y)$. This is illustrated in Figure \ref{flop_fig}. 

\begin{figure}

\begin{center}
\begin{tikzpicture}[scale=0.6]
\draw (0,0.7)--(13,0.7);
\draw (0,3.1)--(13,3.1);
\draw (1.5,0.7)--(1.5,3.1);
\draw (4,0.7)--(4,3.1);
\draw (9.25,0.7)--(9.25,3.1);
\draw (11.75,0.7)--(11.75,3.1);
\draw (2.75,1.75) node {\footnotesize $\FF_+[1]$};
\draw (6.5,1.75) node {\footnotesize $\TT_+$};
\draw (10.5,1.75) node {\footnotesize $\FF_+$};

\draw (13,1.75) node {\footnotesize $\cdots$};
\draw (0.5,1.75) node {\footnotesize $\cdots$};
\draw (17,1.75) node {\footnotesize $\D^b(X_+)$};
\draw [decorate, decoration={brace,amplitude=5pt}] (4,3.4)--(11.75,3.4)
node [midway, above=6pt] {\footnotesize $\Coh(X_+)$};
\draw [decorate, decoration={brace,amplitude=5pt}] (9.25,0.5)--(1.5, 0.5)
node [midway, below=6pt] {\footnotesize $\Per^+(X_+/Y)$};

\draw (0,-7.7)--(13,-7.7);
\draw (0,-5.3)--(13,-5.3);
\draw (1.5,-7.7)--(1.5,-5.3);
\draw (4,-7.7)--(4,-5.3);
\draw (9.25,-7.7)--(9.25,-5.3);
\draw (11.75,-7.7)--(11.75,-5.3);
\draw (2.75,-6.5) node {\footnotesize $\FF_-[1]$};
\draw (6.5,-6.5) node {\footnotesize $\TT_-$};
\draw (10.5,-6.5) node {\footnotesize $\FF_-$};

\draw[->] (17,1)--(17,-5.5);
\draw(17.55,-2.25) node{\footnotesize $\isom$};
\draw (13,-6.5) node {\footnotesize $\cdots$};
\draw (0.5,-6.5) node {\footnotesize $\cdots$};
\draw (17,-6.5) node {\footnotesize $\D^b(X_-)$};
\draw[->] (5.6,-1.7)--(5.6,-2.9);
\draw(6.25,-2.3) node{\footnotesize $\isom$};
\draw [decorate, decoration={brace,amplitude=5pt}] (1.5,-5.1)--(9.25,-5.1)
node [midway, above=6pt] {\footnotesize $\Per^{-}(X_-/Y)$};
\draw [decorate, decoration={brace,amplitude=5pt}] (11.75, -7.9)--(4,-7.9)
node [midway, below=6pt] {\footnotesize $\Coh(X_-)$};

\draw (-3.5,-6.5) node {\footnotesize$X_-$};
\draw (-3.5,-2.375) node {\footnotesize$Y$};
\draw (-3.5,1.75) node {\footnotesize$X_+$};
\draw[->] (-3.5,-5.5)--(-3.5,-3);
\draw[->] (-3.5,1)--(-3.5,-1.5);
\end{tikzpicture}
\end{center}
\caption{Effect of a flop on the derived category.\label{flop_fig}}
\end{figure}
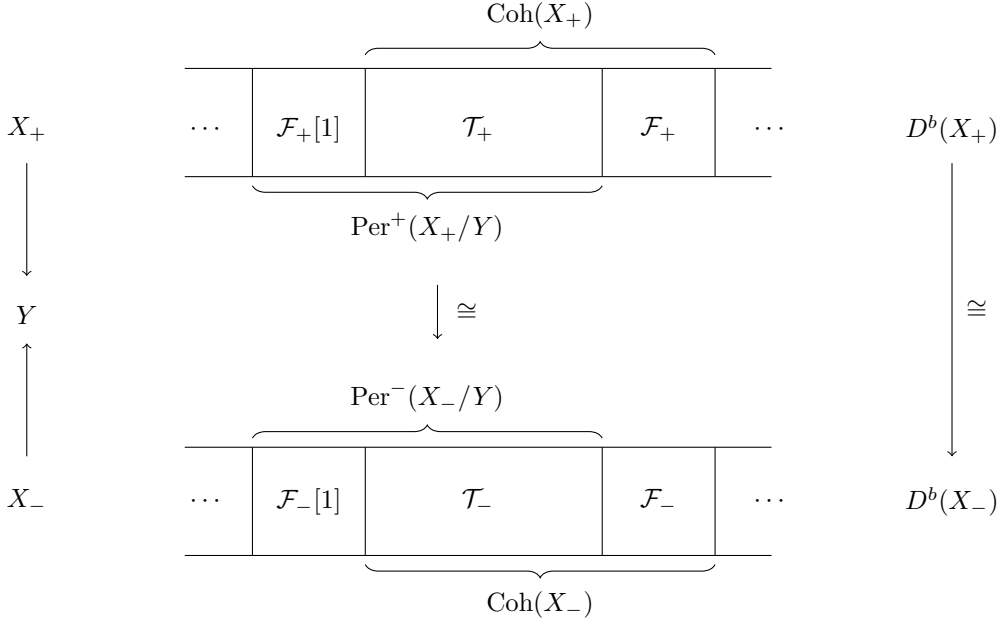

\subsection{Sketch proof of the DT/PT identity.}

Comparing  the identites of Theorems \ref{toda1} and \ref{toda2}  with the tilts described in the last section, one starts to see that one would like to turn the categorical decompositions coming from torsion pairs into identities involving  generating functions of DT invariants. In this section, abandoning all pretence at rigour, we shall explain roughly how this works in the  case of Theorem \ref{toda1}. For a rigorous treatment see  \cite{bridgeland_curve-counting}.

Take notation as in the last subsection.  For any suitable subcategory $\CC$ of $\A$, we consider the elements \[\delta_\CC,\  \delta^\OO_\CC,\  \Quot_\CC\in \hHall_\mot(\A),\]
defined by the stack of objects $E$ of $\CC$, the stack of objects $E$ of $\CC$ equipped with a section $\OO_X\to E$, and the stack of objects $E$ of $\CC$ equipped with a surjective map $\OO_X\to E$, respectively, each  of these stacks being considered with the obvious forgetful map to the stack $\M$ of objects of $\A$.
We will allow ourselves to similarly use elements of the motivic Hall algebra of  $\A^\sharp$, although in reality one can make all calculations in the algebra  $\hHall_\mot(\A)$.

We proceed in three steps:

\begin{itemize}
\item[(i)] Every object $E\in \A$ fits into a unique short exact sequence \[0\lra T\lra E\lra F\lra 0\]
 with $T\in \TT$ and $F\in \FF$. Similarly every $E\in \A^\sharp$ sits in a unique short exact sequence
 \[0\lra F\lra E\lra T[-1]\lra 0.\]
 This gives rise to a torsion pair identities
\[\delta_\A=\delta_\TT * \delta_\FF, \quad \delta_{\A^\sharp}=\delta_\FF*\delta_{\TT[-1]}.\]
Applying $H^0(X,-)$ to the above short exact sequences gives short exact sequences of vector spaces: this is due to cohomology vanishing conditions such as $H^i(X,F)=0$ for $i\notin\{0,1\}$ and $H^i(X,T)=0$ for $i\neq 0$. This gives rise to further identities
 \[\delta^\O_\A=\delta^\O_\TT * \delta^\O_\FF\text{ and } \delta^\O_{\A^\sharp}=\delta^\O_\FF*\delta^\O_{\TT[-1]}.\] 

\item[(ii)]Exactly as in Section \ref{motquot} we have quotient identities
\[\delta^\O_\A  = \Quot^\O_\A * \delta_\A, \quad \delta^\O_{\A^\sharp}  = \Quot^\O_{\A^\sharp} * \delta_{\A^\sharp}, \quad  \delta_\TT^\O =\Quot^\O_\TT* \delta_\TT.\]
On the other hand $H^0(X,T[-1])=0$ implies that $\delta^\O_{\TT[-1]}=\delta_{\TT[-1]}$.
Putting all this together gives \[\Quot^\O_\A * \delta_\TT=\Quot_\TT^\O * \delta_\TT*\Quot^\O_{\A^\sharp}.\]

\item[(iii)]
We have restricted to sheaves supported in dimension $\leq 1$. The Euler form is  trivial so the quantum torus is commutative. Thus
\[\I(\Quot^\O_\A)=\I(\Quot_\TT^\O) * \I(\Quot^\O_{\A^\sharp}).\]
 Setting $t=\pm 1$  then gives the required identity 
\[\sum_{\beta,n} 
\DT(\beta,n)x^\beta y^n=\sum_n\DT(0,n) y^n\cdot \sum_{\beta,n} \PT(\beta,n) x^\beta y^n.\]
\end{itemize}

\smallskip

{\bf Notes.} The tilting operation was introduced in \cite{tilt}. Its application to threefold flops was explained by Van den Bergh \cite{vdb}, following work of the author \cite{bridgeland_flop}. The approach to Theorem 2 sketched above comes from \cite{bridgeland_curve-counting}. A similar proof of Theorem 1 was given by Calabrese \cite{calabrese1}. Toda had previously proved both results for naive DT invariants \cite{toda2,toda6} using Joyce's wall-crossing formula  for rank 1 objects in the derived category. Following technical advances \cite{toda7} his results now also apply to genuine DT invariants.

\bibliography{hall}{}
\bibliographystyle{plain}

\end{document}